\documentclass[12pt,reqno]{amsart}

\usepackage{graphicx}
\usepackage{multicol}
\usepackage{footmisc}
\usepackage{epsfig}
\usepackage{amsmath, amsfonts, amssymb, mathrsfs}
\usepackage{tikz-cd}

\numberwithin{equation}{section}

\newcommand{\Zshat}{\widehat{\Z}_S}

\newcommand{\E}{{\mathbb E}}

\newcommand{\R}{{\mathbb R}}
\newcommand{\C}{{\mathbb C}}
\newcommand{\N}{{\mathbb N}}

\newcommand{\Z}{{\mathbb Z}}

\newtheorem{theo}{{\sc \bf Theorem}}[section]

\newtheorem{lem}[theo]{{\sc \bf Lemma}}
\newtheorem{prop}[theo]{{\sc \bf Proposition}}

\newenvironment{defin}{\medskip\noindent{\bf Definition:\/} }{\medskip}

\begin{document}

\title{Rapid Decay Subalgebras of C$^*$-Algebras}

\author[Hebert]{Shelley Hebert}
\address{Department of Mathematics,
East Mississippi Community College,
8731 S Frontage Rd. Mayhew, MS 39753, U.S.A.}
\email{shebert@eastms.edu}

\author[Klimek]{Slawomir Klimek}
\address{Department of Mathematical Sciences,
Indiana University Indianapolis,
402 N. Blackford St., Indianapolis, IN 46202, U.S.A.}
\email{klimek@iu.edu}

\author[McBride]{Matt McBride}
\address{Department of Mathematics and Statistics,
Mississippi State University,
175 President's Cir., Mississippi State, MS 39762, U.S.A.}
\email{mmcbride@math.msstate.edu}

\date{\today}

\begin{abstract}
We introduce a general scheme of constructing smooth subalgebras of C$^*$-algebras that are closed under the smooth calculus of self-adjoint elements. We illustrate the scheme with a number of examples.
\end{abstract}

\maketitle

\section{Introduction}

There are many places where dense $*$-subalgebras appear in C$^*$-algebra theory and noncommutative geometry. We mention here a few important ones: unbounded derivations \cite{BEJ}, \cite{Bo}, \cite{KMRSW}, cyclic theory \cite{Connes}, \cite{Ren}, deformation-quatization \cite {Rie} and group C$^*$-algebras with the rapid decay (RD) property \cite{Jo}.

By definition, smooth subalgebras of C$^*$-algebras are dense $*$-subalgebras, complete in their own  Fr\'echet topology, and closed under the holomorphic functional calculus. Smooth subalgebras capture a smooth structure in noncommutative geometry. A stronger condition, closedness under the smooth functional calculus of self-adjoint elements is also often considered as part of the definition.

The stability under the smooth functional calculus of self-adjoint elements is harder to prove than the stability under the holomorphic functional calculus. To facilitate such proofs one usually uses the Fourier transform and reduces the stability under the smooth functional calculus of self-adjoint elements to growth estimates of exponentials. The following concept was discussed in \cite{KM2}: we say that a dense Fr\'echet $*$-subalgebra $\mathcal{A}$ of a C$^*$-algebra $A$ has polynomially bounded exponentials (PBE) if for every self-adjoint $a\in\mathcal{A}$ and each of the norms $\|\cdot\|_{\mathcal{A}}$ defining the Fr\'echet topology of $\mathcal{A}$, $\left\| e^{ita}\right\|_{\mathcal{A}}$ is bounded by a polynomial in $t\in\R$. Obviously, the C$^*$-norm of $A$ has polynomially bounded exponentials.

There are several fairly general approaches to  constructions of such smooth algebras, see for example \cite{BC},  \cite{KS} and \cite{BIO}.  However, in some examples below, the norms do not fit those general approaches and different methods are needed to establish the smooth functional calculus of self-adjoint elements.

We propose in this paper a general scheme that produces smooth subalgebras of C$^*$-algebras that are closed under smooth calculus of self-adjoint elements. Our scheme complements the above mentioned constructions and it seems applicable to a different set of examples. It is based on the concept of rapid decay (RD) subalgebras.

Informally, rapid decay subalgebras of C$^*$-algebras are dense $*$-subalgebras for which every element is given by a series with rapid decay coefficients which are further restricted to lie in specific subsets of the C$^*$-algebra. We discuss a precise definition, some general results, and several examples of RD subalgebras of C$^*$-algebras, including odometers, Super Dihedral group C$^*$-algebras, Bunce-Deddens and UHF C$^*$-algebras. For some of those examples we compare the norms obtained from the general scheme with the norms considered in previous papers.

\section{Smooth Subalgebras}

In this section we review a definition and several general results about smooth subalgebras of C$^*$-algebras. More details can be found in \cite{HKM1}. As indicated in the introduction, smooth subalgebras in noncommutative geometry define smooth structures on noncommutative spaces and thus are a critical part of their geometry.

Let $A$ be a unital C$^*$-algebra with norm $\|\cdot\|$ and let  $\mathcal{A}\subseteq A$
be a dense unital $*$-subalgebra of $A$, sharing the same unit as $A$, which is a Fr\'echet $*$-algebra with respect to a locally convex topology stronger than that of $A$. We will always suppose that we can define the Fr\'echet topology of $\mathcal{A}$ using a countable collection of submultiplicative seminorms $\|\cdot\|_N$, $N=0,1,\ldots$. We call $\mathcal{A}$ a {\it Fr\'echet subalgebra} of $A$. The most common way to construct Fr\'echet subalgebras of C$^*$-algebras is as domains of closed derivations and their powers, and more generally as smooth elements of an action of a Lie group, see \cite{Bo}. Powers of derivations are used to define seminorms and give a Fr\'echet topology on such subalgebras.

There are several natural concepts of stability of a subalgebra. We say that $\mathcal{A}$ is {\it spectrally stable} if for any element $a$ of $\mathcal{A}$ its spectrum in $\mathcal{A}$ is the same as its spectrum in $A$. Equivalently, from the definition of the spectrum, $\mathcal{A}$ is spectrally stable if for any element $a\in \mathcal{A}$ that is invertible in $A$ its inverse $a^{-1}$ is in $\mathcal{A}$.

We say that $\mathcal{A}$ is {\it stable under the holomorphic functional calculus} if for any $a\in \mathcal{A}$ and a function $f$ that is holomorphic on an open domain containing the spectrum of $a$ we have $f(a)$ is in $\mathcal{A}$. By considering the holomorphic (away from zero) function $\zeta\mapsto 1/\zeta$, we see that the stability under the holomorphic functional calculus is a stronger condition than the spectral stability. Notice however that if $\mathcal{A}$ is a Fr\'echet subalgebra of $A$ then the spectral stability implies stability under the holomorphic functional calculus.

We say that $\mathcal{A}$ is {\it stable under the smooth functional calculus} of self-adjoint elements if for any self-adjoint element $a$ of $\mathcal{A}$ and a smooth function $f$ defined on an open neighborhood of the spectrum of $a$ we have $f(a)$ is in $\mathcal{A}$. When  $\mathcal{A}$ a Fr\'echet subalgebra of $A$, we have the following elementary observation:
if $\mathcal{A}$ is stable under the smooth functional calculus of self-adjoint elements then $\mathcal{A}$ is stable under the holomorphic functional calculus.

This motivates the following definition: a  dense $*$-subalgebra $\mathcal{A}\subseteq A$ is called {\it smooth} if it is a Fr\'echet subalgebra of $A$ and which is stable under the smooth functional calculus of self-adjoint elements.

In general it is not easy to establish stability under the smooth functional calculus of self-adjoint elements. More information about the Fr\'echet seminorms is required; submultiplicativity is not enough. In a typical approach one can build up any smooth function from exponentials.
This inspires the following concept.
 
We say that $\mathcal{A}$ has {\it polynomially bounded exponentials (PBE)}  if for every self-adjoint $a\in\mathcal{A}$ and every $N=0,1,\ldots$ there are  positive constants $k$ and  $C$ (both depending on $a$ and $N$) such that for every $t\in\R$ we have:
\begin{equation*}
\left\| e^{ita}\right\|_N\leq C(1+|t|)^k.
\end{equation*}

Suppose that a Fr\'echet subalgebra $\mathcal{A}\subseteq A$ has polynomially bounded exponentials. Then $\mathcal{A}$ is closed under the smooth functional calculus of self-adjoint elements; in other words, it is a smooth subalgebra of $A$. 

Finally, we mention here a scenario for establishing the PBE property for a norm that is related in a specific way to another norm that has the PBE property. The following useful general result is inspired by the techniques in \cite{BC}.
\begin{prop}\label{blackcunt} If a norm $\|\cdot\|$ has polynomially bounded exponentials and a norm $\|\cdot\|_1$ satisfies
\begin{equation*}
\|ab\|_1\leq\|a\|\|b\|_1+\|a\|_1\|b\|
\end{equation*}
then then the norm $\|\cdot\|_1$ has polynomially bounded exponentials.
\end{prop}

\section{RD Subalgebras}

In this section, we introduce and discuss a definition of a rapid decay subalgebra $A_\textrm{RD}$ in a C$^*$-algebra $A$, the main object studied in this paper.  Essentially, $A_\textrm{RD}$ is a dense $*$-subalgebra of $A$ where each element is given by an infinite series of rapid decay terms such that each them belongs to a fixed subspace of $A$ with some extra conditions.

We will need a concept of a bimodular projection. Suppose $B$ is a C$^*$-subalgebra of a C$^*$-algebra $A$. A linear continuous map $E:A\to B$ is called a bimodular projection from $A$ to $B$ if $E$ is a projection onto $B$ and has the bimodule property:
\begin{equation*}
E(ab)=E(a)b\ \text{ and }\ E(ba)=bE(a)
\end{equation*}
for all $a\in A$ and $b\in B$. In particular, conditional expectations are examples of bimodular projections.

Let $A$ be a C$^*$-algebra and let $\{A_n\}_{n\in\Z_{\ge0}}$ be a collection of C$^*$-subalgebras of $A$ such that $A_n\subsetneqq A_{n+1}$ and such that $\bigcup_{n\in\Z_{\ge0}}A_n$ is dense in $A$.  For each $n$, suppose that there exist a bimodular projection:
\begin{equation*}
E_n:A_{n+1}\to A_n\,.
\end{equation*}
Also, define maps $E_{m,n}:A_m\to A_n$ for $m>n$ by
\begin{equation*}
E_{m,n} = E_n\circ\cdots\circ E_{m-1}\,.
\end{equation*}
Notice that, by definition, the maps $E_{m,n}$ are bimodular projections.
Additionally, we assume that the maps $E_{m,n}$ are uniformly bounded:
\begin{equation}\label{Emnunifbnd}
\|E_{m,n}\|\leq \Omega
\end{equation}
for some constant $\Omega$. Notice that, since $E_{m,n}$ are projections, we must have $\Omega\ge1$.  Moreover, if $\Omega=1$, by Tomiyama's Theorem (Theorem II.6.10.2 in \cite{Black}), $E_{m,n}$ are conditional expectations.

We can define natural bimodular projections :
\begin{equation*}
\mathcal{E}_n:\bigcup_{m\in\Z_{\ge0}}A_m\to A_n\,,
\end{equation*}
by 
\begin{equation*}
\mathcal{E}_n(a) = E_{m,n}(a)
\end{equation*}
if $a\in A_m$. This is well-defined as $E_{m,n}(a)$ is independent of $m$ because $E_j$'s are projections onto $A_j$'s. Now, since $\bigcup_{m\in\Z_{\ge0}}A_m$ is dense in $A$, then $\mathcal{E}_n$ can be naturally extended to $A$ by the uniform continuity of $E_{m,n}$. 

Let $B_0=A_0$ and for $n\ge1$ set $B_n=\textrm{ker }E_{n-1}$. Then, for each $n$, $B_n$ is a subspace of $A_n$.  Notice that for $a\in A_n$ we can write
\begin{equation*}
a=E_{n-1}(a) + (a-E_{n-1}(a))\,.
\end{equation*}
By definition, $E_{n-1}(a)\in A_{n-1}$, additionally we have that
\begin{equation*}
E_{n-1}(a-E_{n-1}(a)) = E_{n-1}(a)-E_{n-1}^2(a) = E_{n-1}(a)-E_{n-1}(a)=0\,,
\end{equation*}
and so, $a-E_{n-1}(a)\in B_n$.  Thus for $n\ge1$, we have that $A_n\cong A_{n-1}\oplus B_n$.  Proceeding in a recursive fashion we have that
\begin{equation*}
A_n \cong B_0\oplus B_1\oplus \cdots\oplus B_n\,.
\end{equation*}
Thus, if $a\in A_n$, we have that
\begin{equation*}
a=\sum_{j=0}^na_j\,,\quad\textrm{where }a_j\in B_j\,,\textrm{ and }a_0=\mathcal{E}_0(a)\,,a_j=\mathcal{E}_j(a)-\mathcal{E}_{j-1}(a)
\end{equation*}
More generally, we can associate to any $a\in A$ a unique formal series
\begin{equation*}
a=\sum_{j=0}^\infty a_j\,,\quad\textrm{where }a_j\in B_j\,,\textrm{ and }a_0=\mathcal{E}_0(a)\,,a_j=\mathcal{E}_j(a)-\mathcal{E}_{j-1}(a)
\end{equation*}
Since $\mathcal{E}_j$ are bimodular projections, it follows that
\begin{equation*}
\|a_0\|\le \Omega\|a\| \textrm{ and that }\|a_j\|\le 2\Omega\|a\|\textrm{ for }j\ge1\,.
\end{equation*}

We have the following lemma.
\begin{lem}\label{inBnm}
If $a\in B_n$ and $b\in B_m$, then $ab\in B_n$ if $n>m$, $ab\in B_m$ if $m>n$, and $ab\in A_n$ if $n=m$.
\end{lem}
\begin{proof}
If $n>m$, then $A_m\subseteq A_n$, however, $A_n$ is an algebra, thus $ab\in A_n$.  Since $E_{n-1}$ is a bimodular projection and $a\in B_n$, we have that
\begin{equation*}
E_{n-1}(ab) = E_{n-1}(a)b = 0\,.
\end{equation*}
Thus, $ab\in B_n$.  The proof for $n<m$ is the same and if $n=m$, we can only guarantee $ab\in A_n$.
\end{proof}

We are now ready to state a suitable definition of a rapid decay subalgebra.

\begin{defin}
Let $A$ be a C$^*$-algebra and $B_n$'s defined as above, let $\{\lambda_n\}$ be an increasing sequence of numbers such that
\begin{equation}\label{lnassumption}
\lambda_n\geq2\Omega(n+1).
\end{equation}
We define $A_\textrm{RD}$ to be a {\it rapid decay (RD)} subalgebra in $A$ by
\begin{equation*}
A_\textrm{RD} = \left\{a=\sum_{n=0}^\infty a_n: a_n\in B_n\,,\,\|a\|_N := \sum_{n=0}^\infty\|a_n\|\lambda_n^N<\infty\,,N=0,1,\ldots\right\}.
\end{equation*}
\end{defin}

First notice that if $a\in A_\textrm{RD}$, then its series representation is norm convergent in $A$ using the Weierstrass $M$-test and the fact that $\|a\|\leq\|a\|_0=\sum_{n\ge0}\|a_n\|<\infty$.  
The $a_n$'s in this expansion are unique since they are given by the bimodular projections $\mathcal{E}_n$ applied to $a$. 

It is clear that $\|\cdot\|_N$ defines a norm for each $N$, so it follows that $A_\textrm{RD}$ must be a subspace of $A$.  Moreover, since the partial sums of $a$ converge in $\|\cdot\|$, it follows that $A_\textrm{RD}$ is dense in $A$.  
What is not immediately clear is that $\|\cdot\|_N$ is submultiplicative which we verify using the following lemma.

\begin{lem}\label{Nnorm_est_in_A_n}
If $a\in A_n$, then $\|a\|_N\le  \|a\|\lambda_n^{N+1}$
\end{lem}
\begin{proof}
Since $a\in A_n$, we use its finite expansion to get
\begin{equation*}
\|a\|_N=\|a_0+\cdots+a_n\|_N = \sum_{j=0}^n\|a_j\|\lambda_j^N\le 2\Omega \lambda_n^N\sum_{j=0}^n\|a\| = 2\Omega (n+1)\lambda_n^N\|a\|\le  \|a\|\lambda_n^{N+1},
\end{equation*}
where we used that fact that the $\lambda_n$'s form an increasing sequence, the fact that $\|a_j\|\le 2\Omega\|a\|$ and assumption \eqref{lnassumption}.
\end{proof}

We refer to norms $\|\cdot\|_N$ as RD norms.

\begin{prop}
The RD norms $\|\cdot\|_N$ are submultiplicative norms for $N\ge1$.  Moreover we have that
\begin{equation*}
\|ab\|_0\le \|a\|_1\|b\|_0\textrm{ and }\|ab\|_0\le\|a\|_0\|b\|_1\,.
\end{equation*}
\end{prop}
\begin{proof}
For $a,b\in A_\textrm{RD}$ we expand $a$ and $b$ in their respective series representations, and thus by the triangle inequality we have that
\begin{equation*}
\begin{aligned}
\|ab\|_N &= \left\|\sum_{n=0}^\infty a_n\sum_{m=0}^\infty b_m\right\| = \left\|\sum_{n>m}a_nb_m + \sum_{n<m}a_nb_m + \sum_{n=m}a_nb_m\right\|\\
&\le \sum_{n>m}\|a_nb_m\|_N + \sum_{n<m}\|a_nb_m\|_N + \sum_{n=m}\|a_nb_m\|_N
\end{aligned}
\end{equation*}
Using Lemma \ref{inBnm}, notice that for $n>m$, $a_nb_m\in B_n$, for $n<m$, $a_nb_m\in B_m$ and for $n=m$, $a_nb_m\in A_n$.  Thus by Lemma \ref{Nnorm_est_in_A_n} we have that
\begin{equation*}
\begin{aligned}
\|ab\|_N &\le \sum_{n>m}\|a_n\|\|b_m\|\lambda_n^N + \sum_{n<m}\|a_n\|\|b_m\|\lambda_m^N + \sum_n\|a_n\|b_n\|\lambda_n^{N+1} \\
&\le \sum_{n\neq m}\|a_n\|\|b_m\|\lambda_n^N\lambda_m^N + \sum_n\|a_n\|b_n\|\lambda_n^{N+1}\\
&\le \sum_{m,n}\|a_n\|\|b_m\|\lambda_n^N\lambda_m^N = \|a\|_N\|b\|_N\,.
\end{aligned}
\end{equation*}
where the last inequality follows if $N\ge1$.
\end{proof}

Our next goal is to establish the PBE property of the $N$-norms $\|\cdot\|_N$. This is done using a perturbative argument and it is broken into several steps described in the following lemmas.

For any natural $n$ and $a\in A_\textrm{RD}$, write $a=a_{\leq n} + a_{>n}$ where 
\begin{equation*}
a_{\leq n} = \sum_{k\le n}a_k \quad \text{and}\quad a_{> n} = \sum_{k>n}a_k\,,
\end{equation*}
with $a_k\in B_k$. Note that $a_{\leq n}$ is always a finite sum and is an element of $A_n$; we will refer to this as the \textit{head} of $a$ while $a_{> n}$ is the \textit{tail} of $a$.

To control the size of the tail of $a$ in $A_\textrm{RD}$ we need the following estimate.
\begin{lem}\label{lemma1}
For any $\alpha > 0$ we have
$$\|a_{> n}\|_{N} \le \dfrac{1}{\lambda_{n+1}^\alpha} \|a_{> n}\|_{N+\alpha}.$$
\end{lem}
\begin{proof}
\begin{equation*}
\begin{aligned}
\|a_{> n}\|_N &= \sum_{k>n}\|a_k\|\lambda_k^N 
\le \frac{1}{\lambda_{n+1}^\alpha}\sum_{k>n}\|a_k\| \lambda_k^{N+\alpha} = \frac{1}{\lambda_{n+1}^\alpha}\|a_{> n}\|_{N+\alpha}.
\end{aligned}
\end{equation*}
\end{proof}

Next we estimate the exponentials of heads for self-adjoint $a$ in $\mathcal{A}$.
\begin{lem}\label{lemma2}
If $a\in A_\textrm{RD}$ and self-adjoint then 
$$\|e^{ia_{\leq n}}\|_{N }\le \lambda_n^{N+1}.$$
\end{lem}
\begin{proof}
Since $a$ is self-adjoint, it follows that $a_{\le n}$ is self-adjoint.  Moreover, since $A_n$ is an algebra, we have that $e^{ia_{\le n}}\in A_n$ and thus the result follows from Lemma \ref{Nnorm_est_in_A_n}.
\end{proof}

Products of heads and tails are estimated as follows.

\begin{lem}\label{lemma3}
If $a,b\in A_\textrm{RD}$ then 
$$\|a_{\leq n}b_{> n}\|_{N} \le \|a_{\leq n}\|_{0} \|b_{> n}\|_{N}\quad\textrm{and}\quad \|b_{>n} a_{\le n}\|_{N} \le \|a_{\leq n}\|_{0} \|b_{> n}\|_{N}.$$
\end{lem}
\begin{proof}
First observe that
\begin{equation*}
a_{\le n}b_{>n}=\sum_{k=0}^na_k\sum_{l>n}b_l = \sum_{k=0}^n\sum_{l>n}a_kb_l
\end{equation*}
and that $a_kb_l\in B_l$ by Lemma \ref{inBnm}. Thus, by this observation we have that
\begin{equation*}
\|a_{\le n}b_{>n}\|_N =\sum_{k=0}^n\sum_{l>n}\|a_kb_l\|\lambda_l^N\le\sum_{k=0}^n\|a_k\|\sum_{l>n}\|b_l\|\lambda_l^N=\|a_{\le n}\|_0\|b_{>n}\|_N\,.
\end{equation*}
A similar argument works for the other inequality.
\end{proof}

The key step in estimating sums of heads and tails is the use of Trotter's formula which separates the corresponding exponentials.

\begin{lem}\label{lemma4}
If $a\in A_\textrm{RD}$ and self-adjoint then 
\begin{equation*}
\|e^{ita}\|_{N} \le \lambda_n^{N+1} \exp\left(|t|\lambda_n^2 \|a_{> n}\|_{N}\right)
\end{equation*}
for every natural $n$ and $N\ge1$.
\end{lem}
\begin{proof}
By Trotter's formula we have that
\begin{equation*}
e^{ita} = \lim\limits_{j\to\infty} \left(e^{ita_{\leq n}/j}e^{ita_{> n}/j}\right)^j,
\end{equation*}
and note that convergence is in any submultiplicative norm on the algebra.

The terms in Trotter's formula can be written as 
\begin{equation*}\begin{aligned}
\left(e^{ita_{\leq n}/j} e^{ita_{> n}/j}\right)^j &= \left(\prod_{l=1}^j e^{ilta_{\le n}/j}e^{ita_{>n}/j} e^{-ilta_{\le n}/j}\right) e^{ita_{\le n}} \\&=\left(\prod_{l=1}^j \exp\left(\frac{it}{j} e^{ilta_{\le n}/j} a_{>n} e^{-ilta_{\le n}/j}\right)\right) \exp\left(ita_{\le n}\right).
\end{aligned}
\end{equation*}

By submultiplicativeness of the $N$-norms for $N\ge1$, we have that
\begin{equation*}
\begin{aligned}
\left\|\left(e^{ita_{\leq n}/j}e^{ita_{> n}/j}\right)^j\right\|_{N}  &\le
\left(\prod_{l=1}^j \exp\left\|\frac{t}{j} e^{ilta_{\le n}/j} a_{>n} e^{-ilta_{\le n}/j}\right\|_{N}\right) \left\|\exp\left(ita_{\le n}\right)\right\|_{N}.
\end{aligned}
\end{equation*}

By Lemma \ref{lemma3} we get 
\begin{equation*}
\left\| e^{ilta_{\le n}/j} a_{>n} e^{-ilta_{\le n}/j}\right\|_{N} \le \left\|e^{ilta_{\le n}/j}\right\|_{0} \|a_{>n}\|_{N} \left\|e^{-ilta_{\le n}/j}\right\|_{0}.
\end{equation*}

Using Lemma \ref{lemma2}, we then get 
\begin{equation*}
\begin{aligned}
\left\|\left(e^{ita_{\leq n}/j}e^{ita_{> n}/j}\right)^j\right\|_{N}  &\le \lambda_n^{N+1} \exp\left(\frac{|t|}{j} \sum_{l=1}^j \lambda_n^2\|a_{>n}\|_{N}\right)\\
&=\lambda_n^{N+1} \exp\left(|t|\lambda_n^2 \|a_{> n}\|_{N}\right).
\end{aligned}
\end{equation*}
Note that this estimate is independent of $j$ and consequently we get that 
\begin{equation*}
\|e^{ita}\|_{N} \le \lambda_n^{N+1} \exp\left(|t|\lambda_n^2 \|a_{> n}\|_{N}\right).
\end{equation*}
\end{proof}

\begin{theo}
With the above notation, the RD norms $\|\cdot\|_N$ have the PBE property.
\end{theo}

\begin{proof}
By Lemmas \ref{lemma4} and \ref{lemma1}, using $\alpha=3$, we have that
\begin{equation*}
\begin{aligned}
\|e^{ita}\|_N &\le\lambda_n^{N+1}\exp{\left(|t|\lambda_n^2\|a_{>n}\|_N\right)}\le\lambda_n^{N+1}\exp{\left(|t|\frac{\lambda_n^2}{\lambda_{n+1}^3}\|a_{>n}\|_{N+3}\right)}\\
&\le \lambda_n^{N+1}\exp{\left(\frac{|t|}{\lambda_{n+1}}\|a\|_{N+3}\right)}
\end{aligned}
\end{equation*}
where we used the fact that the $\lambda_n$'s are an increasing sequence of numbers and the fact that $\|a_{>n}\|_N\le \|a\|_N$. Notice that this works for any $n$. Now, for $|t|\ge1$, we choose $n$ so that 
\begin{equation*}
\lambda_n\le |t|<\lambda_{n+1}\,.
\end{equation*}
Using this in the above inequality yields
\begin{equation*}
\|e^{ita}\|_N\le|t|^{N+1}e^{\|a\|_{N+3}}
\end{equation*}
which establishes the polynomially bounded exponential property.
\end{proof}

As mentioned in the previous section, the above theorem implies that $A_\textrm{RD}$ are smooth subalgebras of $A$.

Let $\{\lambda_n\}$ and $\{l_n\}$ be two sequences of numbers both increasing and with limits equal to infinity.  Then we can write two different types of RD norms:
\begin{equation*}
\|a\|_N = \sum_{n=0}^\infty \|a_n\|l_n^N\quad\textrm{and}\quad |a|_N =\sum_{n=0}^\infty\|a_n\|\lambda_n^N\,.
\end{equation*}
The natural question that arises is: When are the two families of these norms equivalent?

If the families of these norms are equivalent then there exist a constant $C_1$ and $N'$ depending on $N$ so that $\|a\|_N\le C_1|a|_{N'}$ for all $a\in A_\textrm{RD}$.  This says that
\begin{equation*}
\sum_{n=0}^\infty\|a_n\|l_n^N\le C_1\sum_{n=0}^\infty\|a_n\|\lambda_n^{N'}\,.
\end{equation*}
In particular, if $a_n=a$ for $n=k$ with $\|a\|=1$, and $a_n=0$ else, then $l_k^N\le C_1\lambda_k^{N'}$. So, there exist constants, $C$ and $\alpha$ depending on $N$ such that 
$$l_k\le C\lambda_k^\alpha$$ 
for all $k$.  Similarly, there exists constants $C'$ and $\alpha'$ depending on $N$ and $N'$ such that for all $k$ we have 
$$\lambda_k\le C'l_k^{\alpha'}.$$

On the other hand, if the sequences $\{\lambda_n\}$ and $\{l_n\}$ satisfy the two inequalities above, it follows that $\{\|\cdot\|_N\}$ and $\{|\cdot|_N\}$ are equivalent families of norms. In particular, if one of those families has the PBE property, so does the other. This means in particular that the condition \eqref{lnassumption} can be weaken and still lead to norms with the PBE property.

\section{Examples}
\subsection{Convergent Sequences} 
This is the easiest motivating example. 

A non-unital C$^*$-algebra in this example is the algebra of sequences convergent to zero:
\begin{equation*}
c_0:=\left\{x=(x_k)_{k\in\Z_{\geq0}}:\lim_{k\to\infty}x_k=0\right\}
\end{equation*}
equipped with coordinate-wise addition and multiplication and the supremum norm.

Subalgebras $A_n$ are defined to be the algebras of sequences that are eventually zero:
\begin{equation*}
A_n:=\{x\in c_0: x_k=0 \text{ for }k>n\}
\end{equation*}

Natural conditional expectations $E_n:A_{n+1}\to A_n$ are defined by 
\begin{equation*}
E_n(x)_k = \left\{
\begin{aligned}
&x_k &&\textrm{if }k\leq n  \\
&0 &&\textrm{else.}
\end{aligned}\right.
\end{equation*}
Consequently, subspaces $B_n=\ker E_{n-1}$ are one-dimensional and spanned by delta sequences $\delta_n$ defined by
\begin{equation*}
\delta_{n,k} = \left\{
\begin{aligned}
&1 &&\textrm{if }k=n  \\
&0 &&\textrm{else.}
\end{aligned}\right.
\end{equation*}
It follows that the Fourier decomposition of $x\in c_0$ is
\begin{equation*}
x=\sum_{n=0}^\infty x_n\delta_n\,.
\end{equation*}

We can now easily compute the RD norms:
\begin{equation*}
\|x\|_N=\sum_{n=0}^\infty |x_n|\,\lambda_n^N\,.
\end{equation*}
Thus, with the choice $\lambda_n=1+n$, the corresponding RD algebra coincides with the usual concept of rapid decay sequences.

\subsection{Odometers} 
In this example we compare the general scheme from the previous section with the results from \cite{KM2}. Some of the material in this subsection is relevant for the other examples.

\subsubsection{Definitions}

Let $S$ be an infinite supernatural number. The odometer space $\Z_S$ can be defined as an inverse limit
\begin{equation*}
\Z_S:=\lim_{\underset{s|S}{\longleftarrow}} \Z/ s\Z \,,
\end{equation*}
where the corresponding projection homomorphisms are reductions modulo $s$. It is a compact Abelian group and a Cantor space. In the special case of $S=p^\infty$ for a prime $p$ this is the space of $p$-adic integers \cite{Do}.

The Pontryagin dual space of $\Z_S$ is the corresponding direct limit and can be identified with the following discrete group:
\begin{equation*}
\widehat{\Z_S}=\{z\in\C: \textrm{ there is }s|S \textrm{ such that }z^s=1\},
\end{equation*}
see \cite{HeRo}.

We choose a scale  $s = (s_m)_{m\in\N}$ for the supernatural number $S$, which  is a sequence of positive integers such that $s_m$ divides $s_{m+1}$, $s_m<s_{m+1}$, and such that $S=\mathrm{lcm}(s_m)$. Set $s_0:=1$.
Then, $\widehat{\Z_S}$ is the direct limit of the following finite subgroups of $\widehat{\Z_S}$:
\begin{equation*}
G_m:=\{e^{2\pi ij/s_m}:j=0,\ldots,s_m-1\}.
\end{equation*}

In this example, the C$^*$-algebra $A=C(\Z_S)$ is the unital commutative algebra of continuous functions on $\Z_S$, which is isomorphic to the C$^*$-algebra of the discrete Abelian group $\widehat{\Z_S}$.

For a $z\in\widehat{\Z_S}$, let $\chi_z\in C(\Z_S)$ be the corresponding character function on $\Z_S$.  The subalgebras $A_n$, denoted below by $C_n(\Z_S)$ are defined as functions on $\Z_S$ which are given by a finite Fourier series whose sum is over the characters in $G_n$:
\begin{equation*}
C_n(\Z_S):=\left\{\sum_{z\in G_n}\hat f_z \chi_z: \hat f_z\in\C\right\}.
\end{equation*}
These subaglebras are also used in the following subsections.

Let $dx$ be the normalized Haar measure on $\Z_S$. To any continuous function $f$ on $\Z_S$ one can associate the usual Fourier series:
\begin{equation}\label{FSeries}
\sum_{z\in\widehat{\Z_S}}\hat f_z \chi_z,
\end{equation}
where
\begin{equation*}
\hat f_z=\int_{\Z_S}f(x) \chi_{\bar z}(x)\,dx
\end{equation*}
This series is not always convergent but the Fourier coefficients $\hat f_z$ determine $f$. Using this Fourier series we see that the
subalgebra $C_n(\Z_S)$ can be alternatively characterized as consisting of those continuous functions $f$ such that for every $x\in\Z_S$
\begin{equation*}
f(x+s_n)=f(x).
\end{equation*}

\subsubsection{Conditional Expectations}

Let $\alpha: C(\Z_S)\to C(\Z_S)$ be the automorphism of $C(\Z_S)$ given by 
$$\alpha f(x)=f(x+1).$$

We define $\E_n:C_{n+1}(\Z_S)\to C_{n+1}(\Z_S)$ by
\begin{equation*}
\E_n(f):=\frac{s_n}{s_{n+1}}\sum_{j=0}^{s_{n+1}/s_n-1}\alpha^{js_n}(f).
\end{equation*}
We have the following result.
\begin{prop}\label{cond_exp_cont} $\E_n$ is a conditional expectation onto $C_{n}(\Z_S)$ and
\begin{equation}\label{odoexp}
\E_n\left(\sum_{z\in G_{n+1}}\hat f_z \chi_z\right)=\sum_{z\in G_{n}}\hat f_z \chi_z
\end{equation}

\end{prop}
\begin{proof}
First note that applying $\alpha$ to a character function $\chi_z$ yields 
\begin{equation*}
\alpha \chi_z(x)=\chi_z(x+1) = \chi_z(x)\chi_z(1) = z\chi_z(x).
\end{equation*}
We compute $\E_n$ on the character functions:
\begin{equation*}
\begin{aligned}
\E_n\chi_z &= \dfrac{s_n}{s_{n+1}} \sum_{j=0}^{s_{n+1}{s_n}-1} \alpha^{js_n}\chi_z \\
&=\dfrac{s_n}{s_{n+1}} \sum_{j=0}^{s_{n+1}/s_n-1} z^{js_n}\chi_z = \dfrac{s_n}{s_{n+1}}\left(\sum_{j=0}^{s_{n+1}/s_n-1} z^{js_n}\right)\chi_z.
\end{aligned}
\end{equation*}
If $z^{s_n}=1$, then $z\in G_n$ and so $z^{js_n}=1$.  Then the above is equal to 
\begin{equation*}
\dfrac{s_n}{s_{n+1}} \left(\sum_{j=0}^{s_{n+1}/s_n-1} 1\right) \chi_z = \chi_z.
\end{equation*}
On the other hand, if $z^{s_n}\ne 1$ then $z\in G_{n+1}\setminus G_n$ and so
\begin{equation*}
\begin{aligned}
\E_n\chi_z &= \dfrac{s_n}{s_{n+1}} \left(\dfrac{1-z^{\frac{s_{n+1}}{s_n}s_n}}{1-z^{s_n}}\right)\chi_z \\
&= \dfrac{s_n}{s_{n+1}} \dfrac{1-z^{s_{n+1}}}{1-z^{s_n}}\chi_z = 0.
\end{aligned}
\end{equation*}

We also get the following norm estimate:
\begin{equation*}
\begin{aligned}
\|\E_nf\| = \left\| \dfrac{s_n}{s_{n+1}} \sum_{j=0}^{s_{n+1}/s_n-1} \alpha^{js_n}f\right\| &\le  \dfrac{s_n}{s_{n+1}} \sum_{j=0}^{s_{n+1}/s_n-1} \left\|\alpha^{js_n}f\right\|\\
&= \dfrac{s_n}{s_{n+1}} \sum_{j=0}^{s_{n+1}/s_n-1} \|f\| = \dfrac{s_n}{s_{n+1}} \dfrac{s_{n+1}}{s_n}\|f\| = \|f\|.
\end{aligned} 
\end{equation*}
Tomiyama's Theorem tells us that $\E_n$ is a conditional expectation.
\end{proof}

It follows from formula \eqref{odoexp} that
\begin{equation*}
B_n(S):=\ker \E_{n-1}=\left\{\sum_{z\in G_n\setminus G_{n-1}}\hat f_z \chi_z: \hat f_z\in\C\right\}.
\end{equation*}
Thus, if $f$ is given by the Fourier series \eqref{FSeries}, we have
\begin{equation*}
\|f\|_N=\sum_{n=0}^\infty \|f_n\|\,\lambda_n^N=\sum_{n=0}^\infty\left\|\sum_{z\in G_n\setminus G_{n-1}}\hat f_z \chi_z\right\|\,\lambda_n^N.
\end{equation*}

Our general scheme from the previous section implies that the algebra:
\begin{equation*}
C_\textrm{RD}(\Z_S) := \left\{f=\sum_{n=0}^\infty f_n: f_n\in B_n(S)\,,\,\|f\|_N = \sum_{n=0}^\infty\|f_n\|\lambda_n^N<\infty\,,N=0,1,\ldots\right\}
\end{equation*}
is a smooth subalgebra of $C(\Z_S)$.

\subsubsection{Alternative Norms}
First we summarize relevant results from \cite{KM2}.

Non-Archimedean length functions on $\widehat{\Z_S}$ are defined as functions $\lambda: \widehat{\Z_S}\to [1,\infty)$ satisfying the following conditions:
\begin{enumerate}
\item $\lambda(z)=1$ iff and only if $z=1$ (normalization),
\item $\lambda(z_1z_2)\leq\text{max}\{\lambda(z_1), \lambda(z_2)\}$ (non-Archimedean property),
\item For every $r\geq 1$ the set $\{z\in\widehat{\Z_S}: \lambda(z)\leq r\}$ is finite (growth condition).
\end{enumerate}

The basic examples of length functions are constructed as follows. First choose a scale  $s = (s_m)_{m\in\N}$ and then choose an increasing sequence $l = (\lambda_m)_{m\in\N}$ of numbers in $(1,\infty)$ such that $\lim_{m\to\infty}\lambda_m=\infty$. Set $\lambda_0:=1$. Given $s$ and $l$ the corresponding length function $\lambda_{s,l}$ is defined as follows. If $z=1$ we set $\lambda_{s,l}(1):=\lambda_0=1$. If $z\ne 1$ we set:
\begin{equation*}
\lambda_{s,l}(z)=\lambda_{s,l}\left(e^{\frac{2\pi ij}{s_m}}\right):=\lambda_m,
\end{equation*}
for all $m\in \N$ and $0<j<s_m$ such that $\frac{s_m}{s_{m-1}}$ does not divide $j$. 
The assumption that $\lim_{m\to\infty}\lambda_m=\infty$ implies the growth property of $\lambda_{s,l}$.
It turns out that $\lambda_{s,l}$ encompass all possible non-Archimedean length functions on $\widehat{\Z_S}$.

Alternative norms on the space of series \eqref{FSeries} are defined as follows \cite{KM2}:
\begin{equation*}
\|f\|_N^*:=\sum_{z\in\widehat{\Z_S}}|\hat f_z|\lambda(z)^N.
\end{equation*}

We say that a non-Archimedean length function $\lambda$ {\it grows fast enough} if there are positive numbers $\beta$ and $c$ such that for every $m$ we have:
\begin{equation}\label{fast_enough}
s_m\leq c\,(\lambda_m)^\beta.
\end{equation}

Under this stronger growth condition on length functions, namely that a non-Archimedean length function $\lambda$ grows fast enough, it was proved in \cite{KM2} that the above norms have the PBE property.

We have the following result.
\begin{prop} Assuming that $\lambda$ grows fast enough and the same sequence $(\lambda_m)_{m\in\N}$ is used in the definition of $\{\|\cdot\|_N\}$, the two sets of norms $\{\|\cdot\|_N\}$ and  $\{\|\cdot\|^*_N\}$ are equivalent.
\end{prop}
\begin{proof}
Notice that
\begin{equation*}
\|f_n\|=\left\|\sum_{z\in G_n\setminus G_{n-1}}\hat f_z \chi_z\right\|\leq \sum_{z\in G_n\setminus G_{n-1}}|\hat f_z|.
\end{equation*}
It follows immediately that $\|f\|_N\le \|f\|_N^*$.  

On the other hand, for $z\in G_n\setminus G_{n-1}$ by definition we get
\begin{equation*}
\hat{f}_z = \int_{\Z_S}f_n(x)\overline{\chi_z}(x)\,dx\,,
\end{equation*}
and so $|\hat{f}_z|\le\|f_n\|$. Therefore, we have that
\begin{equation*}
\begin{aligned}
\|f\|_N^*&\le \sum_{n=0}^\infty\sum_{z\in G_n\setminus G_{n-1}}\|f_n\|\,\lambda_n^N = \sum_{n=0}^\infty(s_n-s_{n-1})\|f_n\|\lambda_n^N\\
&\le \sum_{n=0}^\infty s_n\|f_n\|\lambda_n^N \le const\sum_{n=0}^\infty\|f_n\|\lambda_n^{N+\beta} = const\|f\|_{N+\beta}\,.
\end{aligned}
\end{equation*}
Thus, the two families of norms are equivalent.

\end{proof}

\subsection{Super Dihedral Groups}

\subsubsection{Definitions}
The dihedral group $D_n$ is the semidirect product of $\Z/2\Z$ acting on the cyclic group
$\Z/n\Z$ via the automorphism 
\begin{equation*}
\varkappa (g)=g^{-1}
\end{equation*}
for $g\in \Z/n\Z$. 
For any Abelian group $G$, the generalized dihedral group of $G$, denoted Dih$(G)$, is the semidirect product $\mathrm {Dih} (G)=G\rtimes _{\varkappa }\Z/2\Z$ of $G$ and $\Z/2\Z$ acting on the group
$G$ via the above inversion automorphism $\varkappa$. 
Specializing this definition to discrete Abelian groups $\widehat{\Z_S}$, we define super dihedral groups:
\begin{equation*}
\mathrm {Dih} (\widehat{\Z_S})=\widehat{\Z_S}\rtimes _{\varkappa }\Z/2\Z.
\end{equation*}
The main objects of study in this example are the group C$^*$-algebras of super dihedral groups:
\begin{equation*}
SD(S):=C^*(\mathrm {Dih} (\widehat{\Z_S})).
\end{equation*}

We start by discussing the structure of the C$^*$-algebra $SD(S)$. First recall a few facts on amenability of discrete groups. A good source on amenability is \cite{Pat}. 
Finite groups are amenable as are direct limits of amenable groups. In particular, since $\widehat{\Z_S}$ is the direct limit of finite cyclic groups, it is amenable. 
A group extension of an amenable group by an amenable group is again amenable. In particular, super dihedral groups are amenable.

The reduced group C$^*$-algebra of a group $G$ is isomorphic to the non-reduced group C$^*$-algebra defined above if and only if $G$ is amenable. Thus, if $G$ is a super dihedral group, then $$SD(S)=C^*(\mathrm {Dih} (\widehat{\Z_S})) \cong C_r^*(\mathrm {Dih} (\widehat{\Z_S})).$$

By example 3.16 (Semidirect Products of Abelian Groups) of \cite{Williams} we have $SD(S)$ is isomorphic with the crossed product algebra:
\begin{equation*}
SD(S)\cong C(\Z_S)\rtimes_\kappa\Z/2\Z\,.
\end{equation*} 
Here $\kappa: C(\Z_S)\to C(\Z_S)$ is given by
\begin{equation*}
\kappa(f)(x)=f(-x).
\end{equation*}

Explicitly, by universality of crossed products, $SD(S)$ can be described as the universal C$^*$-algebra with generators and relations as follows:
\begin{equation*}
SD(S)\cong C^*(f,v: f\in C(\Z_S), v^*v=vv^*=v^2=I, vfv=\kappa(f)).
\end{equation*}
In particular, we have $v^*=v$.

We define the automorphism $\gamma: SD(S)\to SD(S)$, initially on generators, by
\begin{equation*}
\gamma(f)=f\ \text{ and }\ \gamma(v)=-v.
\end{equation*}
Since it preserves the relations in $SD(S)$ it extends to an automorphism on the whole algebra. 

It follows from the relations that any element $a$ of $SD(S)$ can be written in the form
\begin{equation*}
a=f+vg,
\end{equation*}
with $f,g\in C(\Z_S)$. We can use the automorphism $\gamma$ to see that this representation is unique since
\begin{equation*}
a+\gamma(a)=2f\ \text{ and }\ v(a-\gamma(a))=2g.
\end{equation*}



We define natural subalgebras $A_n$ of $SD(S)$ by 
\begin{equation*}
    A_n=\{f+vg: f,g\in C_n(\Z_S)\}\,.
\end{equation*}
Since the union $\bigcup_{n\in\Z_{\geq0}}C_n(\Z_S)$ is dense in $C(\Z_S)$, the union
$\bigcup_{n\in\Z_{\geq0}}A_n$ is dense in $SD(S)$.

\subsubsection{Bimodular Projections}

Recall that $\E_n:C_{n+1}(\Z_S)\to C_n(\Z_S)$ is a conditional expectation onto $C_n(\Z_S)$.  Define $E_n:A_{n+1}\to A_n$ by
\begin{equation*}
E_n(a) = \E_n(f)+v\E_n(g)\,.
\end{equation*}

We have the following proposition:
\begin{prop}
For any $n$, $E_n$ is a bimodular projection onto $A_n$ and $E_n(a^*)=\left(E_n(a)\right)^*$.  Moreover, if $m>n$, then $\|E_{m,n}\|\le 2$ where $E_{m,n}=E_n\circ\cdots\circ E_{m-1}$.
\end{prop}

\begin{proof}
It is clear that $\E_n(\bar{f}) = \overline{\E_n(f)}$.  Moreover, a direct calculation shows that 
\begin{equation*}
\E_n(\kappa f)=\kappa(\E_n(f))\,.
\end{equation*}
We get
\begin{equation*}
E_n(a^*)=E_n(\bar{f}+v\kappa \bar{g}) = \E_n(\bar{f}) + v\E_n(\kappa \bar{g}) = \left(E_n(a)\right)^*.
\end{equation*}

Let $a=f_1+vg_1$ be in $A_{n+1}$ and $b=f_2+vg_2$ be in $A_n$.  Then using the above fact that $\E_n$ commutes with $\kappa$, we get
\begin{equation*}
\begin{aligned}
E_n(ab) &= E_n\left((f_1+vg_1)(f_2+vg_2)\right) 
= \E_n\left(f_1 f_2 + \kappa(g_1)g_2\right)+v\E_n\left(g_1 f_2 + \kappa(f_1)g_2\right) \\
&= \E_n(f_1)f_2 + \E_n(\kappa(g_1))g_2 + v\left(\E_n(g_1)f_2+\E_n(\kappa(f_1))g_2\right) = E_n(a)b.
\end{aligned}
\end{equation*}
Similarly, if $a\in A_{n}, b\in A_{n+1}$, $E_n(ab)=aE_n(b)$.

Recall that $\gamma:SD(S)\to SD(S)$ defined by $\gamma(f)=f, \gamma(v)=-v$ is an automorphism of $SD(S)$.  If $a=f+vg$ then 
\begin{equation*}
f=\frac{a+\gamma(a)}{2} = \left(\frac{1+\gamma}{2}\right)a\quad\textrm{and}\quad g = v\left(\frac{a-\gamma(a)}{2}\right) = v\left(\frac{1-\gamma}{2}\right)a\,.
\end{equation*}
It follows that $\|f\|_\infty\le \|a\|$ and $\|g\|_\infty\le \|a\|$. 

Then
\begin{equation*}
\|E_{m,n}(a)\| \le \|\E_{m,n}(f)\| + \|\E_{m,n}(g)\| \le \|f\|_\infty+\|g\|_\infty \le 2\|a\|
\end{equation*}
where $\E_{m,n}=\E_n\circ\cdots\circ \E_{m-1}$ for $m>n$.
\end{proof}

It follows from the definition that we have 
\begin{equation*}
B_n=\ker E_{n-1} = \left\{f+vg : f,g\in B_n(S)\right\}
\end{equation*}
where we recall that \begin{equation*}
B_n(S) = \ker \E_{n-1} = \left\{\sum_{z\in G_n\setminus G_{n-1}} \hat{h}_z\chi_z\right\}\,.
\end{equation*}

The RD norms $\|\cdot\|_N$ are given by 
\begin{equation*}
\|f+vg\|_N = \left\|\sum_{n=0}^\infty (f_n+vg_n)\right\|_N := \sum_{n=0}^\infty \|f_n+vg_n\|\lambda_n^N
\end{equation*}
where $f_n,g_n\in B_n(S)$.

Notice that the RD norms defined here fit all of the properties of the general RD norms from Section 3 and thus the norms satisfy the PBE property. Moreover,
\begin{equation*}
SD_\textrm{RD}(S) := \left\{a=\sum_{n=0}^\infty (f_n+vg_n): f_n, g_n\in B_n(S)\,,\,\|a\|_N = \sum_{n=0}^\infty\|f_n+vg_n\|\lambda_n^N<\infty\right\}
\end{equation*}
is a smooth subalgebra of $SD(S)$.

\subsubsection{Alternative Norms}
Consider the following natural norms:
\begin{equation*}
\|a\|_N^{\#}=\|f+vg\|_N^{\#}:=\|f\|^*_N+\|g\|^*_N,
\end{equation*}
where the family $\{\|\cdot\|^*_N\}$ of norms on $C_{RD}(\Z_S)$ was defined in the previous subsection. The norms $\|\cdot\|^{\#}_N$ are submultiplicative and are easier to work with than the norms $\{\|\cdot\|_N\}$ above.

We have the following result.
\begin{prop} Assuming that $\lambda$ grows fast enough and the same sequence $(\lambda_m)_{m\in\N}$ is used in the definition of $\{\|\cdot\|_N\}$, the two sets of norms $\{\|\cdot\|_N\}$ and  $\{\|\cdot\|^{\#}_N\}$ are equivalent.
\end{prop}
\begin{proof}
We have:
\begin{equation*}
\|f+vg\|_N^* = \|f\|_N+\|g\|_N = \sum_{n=0}^\infty \|f_n\|\lambda_n^N + \sum_{n=0}^\infty \|g_n\|\lambda_n^N
\end{equation*}
and so 
$$\|f+vg\|_N = \sum_{n=0}^\infty \|f_n+vg_n\|\lambda_n^N \le \|f+vg\|_N^*.$$

With $\gamma$ defined as above, we also get 
\begin{equation*}
\|f_n\|\le \|f_n+vg_n\| \text{\ and\ } \|g_n\|\le \|f_n+vg_n\|
\end{equation*}
and so $\|f+vg\|_N^* \le 2\|f+vg\|$.
\end{proof}

\subsection{Bunce-Deddens Algebras} 

\subsubsection{Definitions}

First we recall a definition of Bunce-Deddens algebras. We use the description from \cite{KMP2} and the notation established in the subsection on odometers above. In particular, we choose a non-Archimedean length function $\lambda$ that grows fast enough.

Let $H=\ell^2(\Z)$,  $\{E_l\}_{l\in\Z}$ be the canonical basis in $H$, and  $U:H\to H$ be the shift operator on $H$:
\begin{equation*}
UE_l = E_{l+1}\,.
\end{equation*} 

For a continuous function $f\in C(\Z_S)$ we define an operator $M_f:H\to H$ via:
\begin{equation*}
M_fE_l = f(l)E_l\,.
\end{equation*}
Notice that $M_f$ is a diagonal multiplication operator on $H$.  
We have
\begin{equation*}
UM_fU^{-1}=M_{\alpha(f)}
\end{equation*}
where $\alpha: C(\Z_S)\to C(\Z_S)$ is given by 
$$\alpha f(x)=f(x+1).$$

We define the algebra $BD(S)$ to be the C$^*$-algebra generated by operators $U$ and $M_f$:
\begin{equation*}
BD(S)=C^*(U,M_f: f\in C(\Z_S))\,.
\end{equation*} 
$BD(S)$ is isomorphic with the following crossed product algebra:
\begin{equation*}
BD(S)\cong C(\Z_S)\rtimes_\alpha \Z\,.
\end{equation*} 
It is known \cite{Dav} that $BD(S)$ is a simple C$^*$-algebra and it has a unique tracial state.

The subalgebras $A_n$  are defined as finite sums over the characters in $G_n$:
\begin{equation*}
A_n:=\left\{\sum_{z\in G_n} \hat{a}_z(U) M_{\chi_z}:\  \hat{a}_z\in C(S^1)\right\}.
\end{equation*}
They can be alternatively characterized as consisting of those elements of $BD(S)$ that satisfy
\begin{equation*}
U^{s_n}aU^{-s_n}=a.
\end{equation*}

Since $G_n\subsetneqq G_{n+1}$, we have $A_n\subsetneqq A_{n+1}$. We need to verify that the union
$\bigcup_{n\in\Z_{\geq0}}A_n$ is dense in $BD(S)$.  By definition of $BD(S)$, it follows that
\begin{equation*}
\left\{\sum_nU^nM_{f_n}: f_n\in C_\textrm{RD}(\Z_S), \textrm{ finite sums}\right\}
\end{equation*}
is dense in $BD(S)$.  By expanding $f_n\in C_{RD}(\Z_S)$ in a Fourier series from equation \eqref{FSeries}:
\begin{equation*}
f_n=\sum_{z\in\widehat{\Z_S}}\hat f_{n,z} \chi_z,
\end{equation*}
and substituting this into the sum in the above set we get
\begin{equation*}
\sum_n U^nM_{f_n} = \sum_{z\in\widehat{\Z_S}}\sum_n\left(\hat f_{n,z}U^n\right)M_{\chi_z} = \sum_{z\in\widehat{\Z_S}}\hat{a}_z(U)M_{\chi_z}
\end{equation*}
where we define
\begin{equation*}
\hat{a}_z(U):=\sum_{n\in\Z}U^nM_{\hat f_{n,z}}\,.
\end{equation*}
It follows that $\hat{a}_z\in C^\infty(S^1)$ and that $\bigcup_{n\in\Z_{\ge0}} A_n$ is dense in $BD(S)$.

\subsubsection{Conditional Expectations}
We define $E_n:A_{n+1}\to A_{n+1}$ by
\begin{equation*}
E_n(a):=\frac{s_n}{s_{n+1}}\sum_{j=0}^{s_{n+1}/s_n-1}U^{js_n}aU^{-js_n}.
\end{equation*}
We have the following result.
\begin{prop} $E_n$ is a conditional expectation onto $A_n$ and
\begin{equation*}
E_n\left(\sum_{z\in G_{n+1}}\hat{a}_z(U) M_{\chi_z}\right)=\sum_{z\in G_{n}}\hat{a}_z(U) M_{\chi_z}.
\end{equation*}
\end{prop}
\begin{proof}
Let $a\in A_{n+1}$, then we write
\begin{equation*}
a=\sum_{z\in G_{n+1}}\hat{a}_z(U)M_{\chi_z}\,.
\end{equation*}
Given an element of the form $a(U)M_{\chi_z}$, where $a$ is a continuous function on the unit circle, by definition of $E_n$ we have that
\begin{equation*}
E_n\left(a(U)M_{\chi_z}\right) = a(U)M_{\E_n(\chi_z)}\,.
\end{equation*}
However, by Proposition \ref{cond_exp_cont}, we have that
\begin{equation*}
\E_n(\chi_z)=\left\{
\begin{aligned}
&\chi_z &&\textrm{ if }z\in G_n \\
&0 &&\textrm{ if }z\in G_{n+1}\setminus G_n\,.
\end{aligned}\right.
\end{equation*}
Thus if $a\in A_n$, it follows that $E_n(a)=a$ and that $\textrm{ran }E_n=A_n$. Notice that
\begin{equation*}
\|E_n(a)\|\le \frac{s_n}{s_{n+1}}\sum_{j=0}^{s_{n+1}/s_n-1}\|U^{js_n}aU^{-js_n}\|\le \frac{s_n}{s_{n+1}}\sum_{j=0}^{s_{n+1}/s_n-1}\|a\| = \|a\|\,.
\end{equation*}
Again, Tomiyama's Theorem tells us that $E_n$ is a conditional expectation.
\end{proof}
It follows from this proposition that
\begin{equation*}
B_n = \textrm{ker }E_{n-1} = \left\{\sum_{z\in G_n\setminus G_{n-1}} \hat{a}_z(U) M_{\chi_z}:\  \hat{a}_z\in C(S^1)\right\}\,,
\end{equation*}
and
\begin{equation*}
\|a\|_N = \left\|\sum_{z\in \widehat{\Z_S}} \hat{a}_z(U) M_{\chi_z}\right\|_N = \sum_{n\in\Z_{\ge0}}\left\|\sum_{z\in G_n\setminus G_{n-1}} \hat{a}_z(U)M_{\chi_z}\right\|\lambda_n^N\,.
\end{equation*}

\subsubsection{Alternative Norms}
The following norms were introduced in \cite{HKM1} to define a smooth subalgebra $BD_\textrm{RD}(S)$ of $BD(S)$:
\begin{equation*}
\|a\|_{M,N}= \sum_{z\in\Zshat} \|\hat{a}_z\|_{C^M} \lambda(z)^N.
\end{equation*}
It was proved in \cite{HKM1} that those norms have the PBE property.  In what follows we relate those results with the general setup in this paper.

We have the following result.
\begin{prop} Assuming that $\lambda$ grows fast enough, the two sets of norms $\{\|\cdot\|_N\}$ and $\{\|\cdot\|_{0,N}\}$ are equivalent.
\end{prop}
\begin{proof}
For $a\in BS(S)$, recall that
\begin{equation*}
\|a\|_N = \sum_{n=0}^\infty\|a_n\|\lambda_n^N\quad\textrm{and}\quad \|a\|_{0,N}=\sum_{z\in\widehat{\Z}_S}\|\hat{a}_z\|_{C^0(S^1)}\lambda(z)^N
\end{equation*}
where $a_n\in B_n$ and $\hat{a}_z$ are continuous functions on $S^1$.  Notice that
\begin{equation*}
\|a_n\|\le \sum_{z\in G_n\setminus G_{n-1}}\|\hat{a}_z(U)\|\|M_{\chi_z}\| = \sum_{z\in G_n\setminus G_{n-1}}\|\hat{a}_z\|_{C^0(S^1)}\,,
\end{equation*}
and it therefore follows that $\|a\|_N\le \|a\|_{0,N}$.  

Next, for a $w\in G_n\setminus G_{n-1}$, consider the following calculation
\begin{equation*}
\begin{aligned}
\frac{1}{s_n}\sum_{j=0}^{s_n-1}U^j(a_nM_{\overline{\chi}_w})U^{-j} &= \frac{1}{s_n}\sum_{j=0}^{s_n-1}\sum_{z\in G_n\setminus G_{n-1}}\hat{a}_z(U)U^jM_{\chi_{zw^{-1}}}U^{-j}\\
&=\frac{1}{s_n}\sum_{z\in G_n\setminus G_{n-1}}\hat{a}_z(U)\sum_{j=0}^{s_n-1}(zw^{-1})^jM_{\chi_{zw^{-1}}} =\hat{a}_w(U)
\end{aligned}
\end{equation*}
since the sum over $j$ is only nonzero if $z=w$.  Thus we have that
\begin{equation*}
\|\hat{a}_w\|_{C^0(S^1)}=\|\hat{a}_w(U)\|\le\frac{1}{s_n}\sum_{j=0}^{s_n-1}\left\|U^j(a_nM_{\overline{\chi}_w})U^{-j}\right\|\le \frac{1}{s_n}\sum_{j=0}^{s_n-1}\|a_n\| =\|a_n\|\,.
\end{equation*}
Using the above observation we get
\begin{equation*}
\begin{aligned}
\|a\|_{0,N}&=\sum_{z\in\widehat{Z}_S}\|\hat{a}_z\|_{C^0(S^1)}\lambda(z)^N =\sum_{n=0}^\infty\sum_{z\in G_n\setminus G_{n-1}}\|\hat{a}_z\|_{C^0(S^1)}\lambda_n^N
\le \sum_{n=0}^\infty\sum_{z\in G_n\setminus G_{n-1}}\|a_n\|\lambda_n^N \\
&= \sum_{n=0}^\infty(s_n-s_{n-1})\|a_n\|\lambda_n^N\le \sum_{n=0}^\infty s_n\|a_n\|\lambda_n^N \le const \sum_{n=0}^\infty \|a_n\|\lambda_n^{N+\beta} = const\|a\|_{N+\beta}\,,
\end{aligned}
\end{equation*}
which establishes the two families of norms are equivalent.
\end{proof}

As a consequence, this proposition gives an alternative proof that the norms $\|\cdot\|_{0,N}$ from \cite{HKM1} have the PBE property.  Additionally, if $\delta_{\mathbb{L}}$ is the derivation in $BD(S)$ defined on generators by
\begin{equation*}
\delta_{\mathbb L}(U)=U\quad\textrm{and}\quad \delta_{\mathbb L}(M_f)=0\,,
\end{equation*}
then 
\begin{equation*}
\|a\|_{M,N} = \sum_{j=0}^M\left(
\begin{array}{c}
M \\ j
\end{array}\right)\|\delta_{\mathbb L}^{j}(a)\|_{0,N}\,.
\end{equation*}
It follows from Proposition \ref{blackcunt} that the norms $\|\cdot\|_{M,N}$ have the PBE property and so
\begin{equation*}
BD_\textrm{RD}(S) := \left\{a\in BD(S)\,,\,\|a\|_{M,N}<\infty,\,M,N=0,1,\ldots\right\}
\end{equation*}
is a smooth subalgebra of $BD(S)$.

\subsection{Uniformly Hyperfinite Algebras} 
\subsubsection{Definitions}
Uniformly Hyperfinite Algebras (UHF) were introduced by James Glimm  in  \cite{Gl}. 

A UHF algebra is a unital C$^*$-algebra $A$ for which there exists a chain 
$$A_0 \subseteq A_1 \subseteq \cdots\subseteq A$$
of unital subalgebras (i.e. subalgebras containing the unit of $A$) such that each $A_n$ is $*$-isomorphic to a matrix algebra $M_{s_n} (\C)$, and the union
$\cup_{n\in\Z_{\geq0}}A_n$ is dense in $A$. We assume that $s_0=1$ and $A_0$ is the one dimensional subalgebra generated by the identity of $A$.

It follows from the structure of unital homomorphisms of matrix algebras that $s_m$ divides $s_{m+1}$, $s_m<s_{m+1}$, and define a supernatural number $S$ by $S=\mathrm{lcm}(s_m)$. Thus, the sequence $(s_m)_{m\in\Z_{\geq0}}$ is a scale for $S$.
Glimm showed that the supernatural number $S$ is a complete invariant of UHF C$^*$-algebras. 

We will denote a UHF algebra with supernatural number $S$ by $UHF(S)$.
It is known \cite{Dav} that $UHF(S)$ is a simple C$^*$-algebra and it has a unique tracial state constructed from normalized traces on matrix subalgebras.

It is often useful to consider a concrete representation of $UHF(S)$. Let $\{E_k\}_{k\in\Z}$ be the standard orthonormal basis in the Hilbert space $H=\ell^2(\Z)$.

For $n=0,1,2, \ldots$ and $0\leq x,y<s_n$ define
\begin{equation*}
P_{(n,x,y)}E_k = \left\{
\begin{aligned}
&E_{k+x-y} &&\textrm{if }s_n|(k-y)  \\
&0 &&\textrm{else.}
\end{aligned}\right.
\end{equation*}
Main properties of operators $P_{(n,x,y)}$ are summarized in the following statement.
\begin{prop}\label{pnxyprop}
With the above notation, the operators $P_{(n,x,y)}$ have the following properties for all  $n=0,1,2, \ldots$ and $0\leq x,y,z,w<s_n$:
\begin{enumerate}
\item $P_{(n,x,y)}^*=P_{(n,y,x)}$
\item  $P_{(n,x,y)}P_{(n,w,z)}=\delta_{y,w}P_{(n,x,z)}$
\item  $\sum\limits_{x=0}^{s_n-1}P_{(n,x,x)}=I$
\item   $\sum\limits_{j=0}^{s_{n+1}/s_n-1}P_{(n+1,x+js_n,y+js_n)}=P_{(n,x,y)}$
\end{enumerate}
\end{prop}

\begin{proof}
The first item follows the observation that the inner product $(E_l,P_{(n,x,y)}E_k)$ is not zero only if $l=k+x-y$ and 
$s_n|(k-y)$, which is equivalent to $k=l+y-x$ and $s_n|(l-x)$.

For the second item, consider the following calculation:
\begin{equation*}
\begin{aligned}
P_{(n,x,y)}P_{(n,w,z)}E_k &= P_{(n,x,y)}\left\{
\begin{aligned}
& E_{k+w-z}  &\textrm{ if }s_n|(k-z)\\
&0 &\textrm{else} 
\end{aligned}\right. \\
&= \left\{
\begin{aligned}
& E_{k+w-z+x-y}  &\textrm{ if }s_n|(k-z)\textrm{ and }s_n|(k+w-z-y)\\
&0 &\textrm{else.} 
\end{aligned}\right.
\end{aligned}
\end{equation*}
Since $s_n|(k-z)$, it follows that $s_n|(w-y)$, but since $0\le w,y<s_n$, it must be that $w=y$ to have a nonzero outcome.  In other words
\begin{equation*}
P_{(n,x,y)}P_{(n,w,z)} = \delta_{wy}P_{(n,x,z)}\,.
\end{equation*}

For the third item notice that $P_{(n,x,x)}E_k = E_k$ only if $s_n|(k-x)$, otherwise it is zero.  Thus
\begin{equation*}
\sum_{x=0}^{s_n-1}P_{(n,x,x)}E_k = E_k = IE_k\,,
\end{equation*}
because there is only one nonzero term in the sum, namely when $x = k\textrm{ mod }s_n$.

For the last item, notice that
\begin{equation*}P_{(n+1,x+js_n,y+js_n)}E_k= E_{k+x-y}
\end{equation*}
only if $s_{n+1}|(k-y-s_n)$ and zero otherwise.   Like in the previous item, there is only one term in the sum below that produces a nonzero term, namely when $j=(k-y)/s_n\textrm{ mod }s_{n+1}/s_n$.  It follows that
\begin{equation*}
\sum_{j=0}^{s_{n+1}/s_n-1}P_{(n+1,x+js_n,y+js_n)}E_k=\left\{
\begin{aligned}
& E_{k+x-y} &&\textrm{ if }s_n|(k-y) \\
&0 &&\textrm{ else}
\end{aligned}\right. = P_{(n,x,y)}E_k\,.
\end{equation*}
\end{proof}

For a fixed $n$ let $A_n$ be the C$^*$-algebra in $B(H)$ generated by $P_{(n,x,y)}$ with $0\leq x,y<s_n$ and let $A$ be the C$^*$-algebra in $B(H)$ generated by all $P_{(n,x,y)}$ with $n=0,1,2, \ldots$ and $0\leq x,y<s_n$, so that $\cup_{n\in\Z_{\geq0}}A_n$ is dense in $A$. By item 1 and 2 of the above proposition, $P_{(n,x,y)}$'s are matrix units and hence $A_n$ is $*$-isomorphic to the matrix algebra $M_{s_n} (\C)$. By item 3, the subalgebras $A_n$ are unital subalgebras of $A$ and by item 4 we have $A_n\subseteq A_{n+1}$ for all $n$. It follows that $A\cong UHF(S)$.

Operators $P_{(n,x,x)}$ are diagonal operators in $H$ with diagonal elements that are periodic with periods $s_n$. It follows 
that the Abelian C$^*$-subalgebra of $UHF(S)$ generated by all $P_{(n,x,x)}$ operators is isomorphic to $C(\Z_S)$.
To define RD norms we choose an increasing sequence $l = (\lambda_m)_{m\in\N}$ of numbers in $(1,\infty)$ such that $\lim_{m\to\infty}\lambda_m=\infty$. Choosing a scale  $s = (s_m)_{m\in\N}$ and $l$ amounts to choosing a length function $\lambda_{s,l}$ on $\widehat{\Z_S}$.

\subsubsection{Conditional Expectations}
We now want to construct conditional expectations $E_n:A_{n+1}\to A_n$. This is done with the help of the following projections in $A_{n+1}$:
\begin{equation*}
P_n:=\sum\limits_{x=0}^{s_n-1}P_{(n+1,x,x)}.
\end{equation*}
We define $E_n:A_{n+1}\to A_n$ by
\begin{equation*}
E_n\left(\sum_{0\le x,y<s_{n+1}} a_{(x,y)}P_{(n+1,x,y)}\right) = \sum_{0\le x,y<s_n} a_{(x,y)} P_{(n,x,y)}.
\end{equation*}
We have the following result.
\begin{prop} $E_n$ is a conditional expectation onto $A_n$.
\end{prop}
\begin{proof}
Note that 
\begin{equation*}
P_{(n,x,y)}=\sum_{j=0}^{s_{n+1}/{s_n}-1} P_{(n+1,x+js_n,y+js_n)}
\end{equation*}
and so
\begin{equation*}
E_n(P_{(n,x,y)}) = E_n\left(\sum_{j=0}^{s_{n+1}/{s_n}-1} P_{(n+1,x+js_n,y+js_n)}\right) = P_{(n,x,y)}
\end{equation*}
since $x+js_n$ and $y+js_n$ are larger than $s_n$ for $j>0$.  Thus $E_n$ is a projection onto $A_n$.

Note that
\begin{equation*}
P_n P_{(n+1,x,y)} P_n = \sum_{j,k=0}^{s_n-1} P_{(n+1,j,j)} P_{(n+1,x,y)} P_{(n+1,k,k)} = \begin{cases}
P_{(n+1,x,y)} & \text{if\ } 0\le x,y<s_n \\ 0 & \text{else}\end{cases}.\end{equation*}

Let 
\begin{equation*}
a=\sum_{0\le x,y<s_{n+1}} a_{(x,y)}P_{(n+1,x,y)}
\end{equation*}
and apply $P_n$ to the left and right side of $a$:
\begin{equation*}
P_n a P_n = \sum_{0\le x,y<s_n} a_{(x,y)}P_{(n+1,x,y)}\,.
\end{equation*}
From this we get the following norm estimate
\begin{equation*}
\left\| \sum_{0\le x,y<s_n} a_{(x,y)} P_{(n+1,x,y)}\right\| \le \|a\|\,.
\end{equation*}

Since $P_{(n+1,x,y)}$ and $P_{(n,x,y)}$ are matrix units, it follows that
\begin{equation*}
\|E_n(a)\|=\left\|\sum_{0\le x,y<s_n} a_{(x,y)}P_{(n,x,y)}\right\|=\left\|[a_{(x,y)}]_{x,y=0}^{s_n-1}\right\| = \left\| \sum_{0\le x,y<s_n} a_{(x,y)} P_{(n+1,x,y)}\right\|\le\|a\|
\end{equation*}
where $\|[a_{(x,y)}]_{x,y=0}^{s_n-1}\|$ is the matrix norm. Thus, by Tomiyama's Theorem, $E_n$ is a conditional expectation. 
\end{proof}

Let $\mathfrak{S}_0=(0,0)$ and let 
$$\mathfrak{S}_n=\{(x,y)\in \Z_{\geq 0}\times \Z_{\geq 0}: s_{n-1}\le x<s_n\textrm{ or }s_{n-1}\le y<s_n\}.$$
Notice that for a fixed $n$, the sets $\{(x,y)\in \Z_{\geq 0}\times \Z_{\geq 0}: 0\le x,y<s_{n-1}\}$ and $\mathfrak{S}_n$ are disjoint, and their union is the set $\{(x,y)\in \Z_{\geq 0}\times \Z_{\geq 0}: 0\le x,y<s_{n}\}$. Thus by a recurrence argument, the sets  $\mathfrak{S}_n$ are disjoint and hence \begin{equation*}
\bigcup_{n=0}^\infty \mathfrak{S}_n=\Z_{\geq 0}\times \Z_{\geq 0}.
\end{equation*}

It follows from the definition of $E_n$ that
\begin{equation*}
B_n = \textrm{ker }E_{n-1} = \left\{\sum_{(x,y)\in\mathfrak{S}_n} a_{(x,y)} P_{(n,x,y)}:a_{(x,y)}\in\C\right\}\,.
\end{equation*}
From our definition of $A_{\textrm{RD}}$ we have that
\begin{equation*}
UHF_{\textrm{RD}}(S)=\left\{a=\sum_{n=0}^\infty\left(\sum_{(x,y)\in\mathfrak{S}_n}a_{(x,y)}P_{(n,x,y)}\right): \|a\|_N<\infty,\,a_{(x,y)}\in\C\right\}
\end{equation*}
where
\begin{equation*}
\|a\|_N=\sum_{n=0}\left\|\sum_{(x,y)\in\mathfrak{S}_n}a_{(x,y)}P_{(n,x,y)}\right\|\lambda_n^N\,.
\end{equation*}

\subsubsection{Alternative Norms}
We can define a more natural family of norms.  For $a\in UHF_{\textrm{RD}}(S)$ define
\begin{equation*}
\|a\|_N^\% = \sum_{n=0}^\infty \sum_{(x,y)\in\mathfrak{S}_n}|a_{(x,y)}|\lambda_n^N\,.
\end{equation*}
As with the other examples we demonstrate that these collections of norms are equivalent.

\begin{prop}
Assuming that $\lambda$ grows fast enough, the two sets of norms $\{\|\cdot\|_N\}$ and $\{\|\cdot\|_N^\%\}$ are equivalent.
\end{prop}
\begin{proof}
It follows immediately from the definitions of $\|\cdot\|_N$ and $\|\cdot\|_N^\%$ that $\|a\|_N\le \|a\|_N^\%$ for $a\in UHF_{\textrm{RD}}(S)$.  On the other hand for $(x,y)\in\mathfrak{S}_n$ we have that 
\begin{equation*}
|a_{(x,y)}|\le \left\|\sum_{(k,l)\in\mathfrak{S}_n}a_{(k,l)}P_{(n,k,l)}\right\|\,.
\end{equation*}
Therefore, we get that
\begin{equation*}
\begin{aligned}
\|a\|_N^\% &\le\sum_{n=0}^\infty\sum_{(x,y)\in\mathfrak{S}_n}\left\|\sum_{(k,l)\in\mathfrak{S}_n}a_{(k,l)}P_{(n,k,l)}\right\|\lambda_n^N = \sum_{n=0}^\infty (s_n^2-s_{n-1}^2)\left\|\sum_{(k,l)\in\mathfrak{S}_n}a_{(k,l)}P_{(n,k,l)}\right\|\lambda_n^N\\
&\le \sum_{n=0}^\infty s_n^2\left\|\sum_{(k,l)\in\mathfrak{S}_n}a_{(k,l)}P_{(n,k,l)}\right\|\lambda_n^N \le const\|a\|_{N+2\beta}\,.
\end{aligned}
\end{equation*}
\end{proof}
It now follows that the family of norms $\{\|\cdot\|_N^\%\}$ have the PBE property.

\end{document}